\documentclass[reqno]{amsart}
\usepackage[backend=bibtex, sorting=none, bibstyle=alphabetic, citestyle=alphabetic, sorting=nyt, maxbibnames=99, giveninits=true, isbn=false, url=false, doi=false, eprint=false]{biblatex}  
\renewbibmacro{in:}{}
\bibliography{references.bib}
\usepackage{amssymb, calrsfs, graphics, graphicx, enumerate, enumitem, url, xcolor, hyperref}
\usepackage[ruled, lined, linesnumbered, longend]{algorithm2e}
\usepackage[justification=centering]{caption}
\usepackage{tikz}
\usetikzlibrary{positioning, arrows, decorations.markings, calc}
\newtheorem{theorem}{Theorem}[section]
\newtheorem{lemma}[theorem]{Lemma}

\numberwithin{equation}{section}

\theoremstyle{remark}
\newtheorem*{remark}{Remark}

\title{On the generalised Dirichlet divisor problem}
\author[C. Bellotti]{Chiara Bellotti}
\address{School of Science\\
The University of New South Wales, Canberra, Australia}
\email{c.bellotti@adfa.edu.au}
\author[A. Yang]{Andrew Yang}
\address{School of Science\\
The University of New South Wales, Canberra, Australia}
\email{andrew.yang1@adfa.edu.au}
\date\today
\keywords{Generalised Dirichlet divisor problem, Karatsuba constant, Riemann zeta-function.}
\subjclass[2020]{Primary 11N56, 11N37 Secondary 11M06}

\begin{document}

\begin{abstract}
We improve unconditional estimates on $\Delta_k(x)$, the remainder term of the generalised divisor function, for large $k$. In particular, we show that $\Delta_k(x) \ll x^{1 - 1.889k^{-2/3}}$ for all sufficiently large fixed $k$.
\end{abstract}

\maketitle

\section{Introduction}
For integer $k \ge 2$, let $d_k(n)$ denote the number of ways that $n$ can be written as the product of exactly $k$ factors. Partial sums of $d_k(n)$ are known to satisfy the asymptotic formula 
\[
\sum_{n \le x}d_k(n) = xP_{k - 1}(\log x) + \Delta_k(x)
\]
where $P_{k - 1}(t)$ is a degree $k - 1$ polynomial, and $\Delta_k(x)$ is a remainder term. The generalised Dirichlet divisor problem, concerning the order of the quantity $\Delta_k(x)$ as $x\to\infty$, is an open problem that has attracted much interest in analytic number theory. It has been conjectured that $\Delta_k(x) \ll_{\varepsilon} x^{1/2 - 1/(2k) + \varepsilon}$ for any $\varepsilon > 0$, a result that implies the well-known Lindel\"of Hypothesis \cite[Ch. 13]{titchmarsh_theory_1986}. 

While the true order of $\Delta_k(x)$ is currently unknown, substantial partial progress has been made. We briefly review two types of results, of the form 
\begin{equation}\label{estimate_alpha}
\Delta_k(x) \ll_{\varepsilon} x^{\alpha_k + \varepsilon},
\end{equation}
and
\begin{equation}\label{estimate_beta}
\|\Delta_k(x)\|_2 := \left(\frac{1}{x}\int_1^x\Delta_k^2(y)\text{d}y\right)^{1/2} \ll_{\varepsilon} x^{\beta_k + \varepsilon}.
\end{equation}
For the first type of bound, it is known for instance that $\alpha_2 \le 517/1648$ \cite{bourgain_mean_2017}, and various bounds are also known for small $k > 2$ \cite{kolesnik_estimation_1981, ivic_ouellet_1989, ivic_2003}. Much less is known for large $k$, with the current best-known bounds taking the form $\alpha_k \le 1 - Dk^{-2/3}$ for some constant $D > 0$, known also as the Karatsuba constant \cite{kolpakova_2011}. In particular, it has been shown that if the Riemann zeta-function $\zeta(s)$ satisfies Richert's bound \cite{richert_zur_1967}, of the form
\begin{equation}\label{condition}
    |\zeta(\sigma + it)| \ll t^{B(1 - \sigma)^{3/2}}\log^{2/3}t,
\end{equation}
uniformly for $1/2 \le \sigma \le 1$ and some constant $B > 0$, then there exists $c_0, c_1 > 0$ for which $\alpha_k \le 1 - c_0(Bk)^{-2/3}$ and $\beta_k \le 1 - c_1(Bk)^{-2/3}$, for sufficiently large fixed $k$. The constant $B$ appearing in \eqref{condition} has been successively refined via consideration of Vinogradov's integral, with the current best known value being $B = 4.45$ due to Ford \cite{ford_vinogradovs_2002}. Karatsuba \cite{karacuba_uniform_1972} first showed that $D = 2^{-5/3}B^{-2/3} \approx 0.116$ (upon taking $B = 4.45$). Two subsequent results by Fujii \cite{fujii_problem_1976} and Panteleeva \cite{panteleeva1988dirichlet} claim $D = 2^{-1/2}(\sqrt{8} - 1)^{-1/3}B^{-2/3} \approx 0.214$ and $D = 2^{-2/3}B^{-2/3} \approx 0.232$ respectively, however both proofs contain errors (see \cite{ford_vinogradovs_2002} for a discussion). In 1989, Ivi\'c and Ouellet \cite{ivic_ouellet_1989} improved the constant for $k > 10$, showing that $D = 2^{2/3}B^{-2/3}/3 \approx 0.196$ and \eqref{estimate_beta} holds with $c_1 = 2/3$. In 2011, Kolpakova \cite{kolpakova_2011} further improved these estimates for $k \ge 186$\footnote{In \cite{kolpakova_2011} it was claimed that the same result holds for all $k \ge 93$, however since \cite[Thm.\ 5]{kolpakova_2011} is ultimately used to bound $\sigma_{k/2}$ instead of $\sigma_k$, we in fact require $k \ge 186$. } to 
\[
D = \left(\frac{2}{3B(1 - 159.9k^{-1})}\right)^{2/3},\qquad c_1 = \left(\frac{5}{6(1 - 79.95k^{-1})}\right)^{2/3},
\]
where, in particular, $D \to (2/3)^{2/3}B^{-2/3}\approx 0.282$ as $k\to\infty$. However, the argument uses the bound $\zeta(\sigma + it) \ll t^{4.45(1 - \sigma)^{3/2}}$ uniformly for $\sigma \in (0.9, 1)$, which is stronger than any known estimates when $\sigma$ is sufficiently close to 1.\footnote{\label{fnote2}In particular, the estimate $\zeta(\sigma + it) \ll t^{B(1 - \sigma)^{3/2}}$ is invalid as $\sigma \to 1$ since $\zeta(1 + it) = \Omega(\log\log t)$.} 

The value of $B$ is a major obstacle to further improvements in the bounds for $\alpha_k$ and $\beta_k$. In a recent breakthrough, Heath-Brown \cite{heathbrown_new_2017} has shown that 
\begin{equation}\label{heathbrown_assumption}
\zeta(\sigma + it) \ll_{\varepsilon} t^{B(1 - \sigma)^{3/2} + \varepsilon},\qquad 1/2 \le \sigma \le 1,
\end{equation}
holds with $B = 8\sqrt{15}/63 = 0.4918\ldots$. While the value of $B$ is substantially reduced from $4.45$, replacing the factor of $\log^{2/3}T$ with $T^{\varepsilon}$ for some fixed $\varepsilon > 0$ is problematic when $\sigma \to 1$. In particular, \eqref{heathbrown_assumption} cannot be directly combined with the argument of \cite{kolpakova_2011} to obtain an improved Karatsuba constant. Nevertheless, by adapting the method of \cite{ivic_ouellet_1989}, Heath-Brown proved the following estimate 
\[
\Delta_k(x) \ll x^{1 - 0.849k^{-2/3}},
\]
i.e. $D = 0.849$, which is currently the sharpest known bound on $\Delta_k(x)$ for every $k\geq 2$. 

In this article we make two contributions. First, we adapt the method of \cite{kolpakova_2011} to use Heath-Brown's estimate \eqref{heathbrown_assumption} and thus benefit from the improved $B$ constant. In doing so, we obtain a new Karatsuba constant of $D = (2/3)^{2/3}B^{-2/3} \approx 1.224$. Second, we develop a new method of estimating $\Delta_k(x)$ that directly uses an exponential sum estimate instead of a bound on $\zeta(s)$. Combined with an improved estimate of such exponential sums, we obtain a Karatsuba constant of $D \approx 1.889$ for sufficiently large $k$. 

We remark that although the second estimate is sharper, the first estimate has two benefits. First, the explicit dependence on $B$ means any improvement in this constant directly translates to an improved Karatsuba constant. Second, the result holds for all fixed $k \ge 30$ instead of for sufficiently large $k$.

\begin{theorem}\label{th1}
Let $k$ be a fixed positive integer. Then, for $k \ge 30$
\[
\Delta_k(x) \ll x^{1 - 1.224(k - 8.37)^{-2/3}}.
\]
Furthermore, for $k \ge 15$,
\[
\|\Delta_k(x)\|_2 \ll x^{1 - 1.421(k - 4.18)^{-2/3}}.
\]
\end{theorem}

More generally, we show that 
\begin{theorem}\label{alpha_theorem}
Suppose $\zeta(\sigma + it)\ll_{\varepsilon} t^{B(1 - \sigma)^{3/2} + \varepsilon}$ uniformly for $1/2 \le \sigma \le 1$ and any $\varepsilon > 0$. Let $\theta \in [1/2, 1)$ be an arbitrary constant. Then, for any fixed integer $k \ge 2k_0(\theta)$ we have 
\[
\alpha_k \le 1 - \left(\frac{2}{3B(k - 2k_1(\theta))}\right)^{2/3},\qquad \beta_k \le 1 - \left(\frac{5}{6B(k - k_1(\theta))}\right)^{2/3},
\]
where 
\begin{equation}\label{k0k1_defn}
k_0(\theta) := \frac{24\theta - 9}{2(4\theta - 1)(1 - \theta)},\qquad k_1(\theta) := k_0(\theta) - \frac{1}{3B(1 - \theta)^{3/2}}.
\end{equation}
\end{theorem}

The result \eqref{heathbrown_assumption} relies primarily on the following exponential sum estimate, due to \cite[Thm.\ 4]{heathbrown_new_2017}
\begin{equation}\label{hb_c_estimate}
\sum_{N < n \le 2N}n^{-it} \ll_{\varepsilon} N^{1 - c(\log t / \log N)^{2} + \varepsilon},\qquad 1 \le N \ll t^{1/2},
\end{equation}
where $c = 49/80$. In Lemma \ref{imp_exponential_sum_est} we refine this estimate by replacing the constant $49/80$ with $1 - 3\log N / \log t$. By combining this estimate with the mean value theorem for Dirichlet polynomials, the following improvement can be obtained.
\begin{theorem}\label{ctheorem}
Let $\delta > 0$. For all fixed integers $k \ge A\delta^{-3}$, where $A$ is an absolute constant, we have
\[
\alpha_k \le 1 - \left(\frac{3}{2^{2/3}} - \delta\right)k^{-2/3}.
\]
In particular,
\[
\Delta_k(x) \ll x^{1 - 1.889k^{-2/3}}
\]
for sufficiently large $k$. 
\end{theorem}

\subsection[Omega results and sign changes]{Omega results and sign changes of $\Delta_k(x)$}
For completeness we briefly compare Theorem \ref{th1} to the best possible results. It has been conjectured \cite{ivic_ouellet_1989} that 
\[
\alpha_k = \beta_k = \frac{1}{2} - \frac{1}{2k},\qquad k \ge 2.
\]
The conjecture involving $\beta_k$ is equivalent to the Lindel\"{o}f Hypothesis, while the conjecture involving $\alpha_k$ implies the Lindel\"{o}f Hypothesis. Currently, the best-known $\Omega$-results concerning $\Delta_k(x)$ are much closer to the conjectured truth than $O$-results. It is known unconditionally that $\Delta_k(x) = \Omega\left(x^{1/2 - 1/(2k)}\right)$ for all $k \ge 2$, i.e.\ that $\alpha_k \ge 1/2 - 1/(2k)$ \cite{hardy_dirichlet_1917, hafner_new_1981, 10.1155/S1073792803130309}. Recent developments have focused on lower-order factors, in particular Soundararajan \cite{10.1155/S1073792803130309} has shown that
\[
\Delta_k(x)=\Omega\left((x\log x)^{\frac{k-1}{2k}}(\log_2 x)^{\frac{k+1}{2k}(k^{2k/k+1}-1)}(\log_3 x)^{-\frac{1}{2}-\frac{k-1}{4k}}\right)
\]
for all $k \ge 2$. Here, $\log_j x$ represents the $j$th iterated logarithm. 

Another type of result of interest is the frequency of sign changes of $\Delta_k(x)$. In 1955, Tong \cite{Tong1955ONDP} proved that for $k \ge 2$, $\Delta_k(x)$ changes sign at least once in the interval $[X, X + h_k]$ for $h_k \gg_k X^{1 - 1/k}$. 
In the case $k = 2$, Heath-Brown and Tsang \cite{HEATHBROWN199473} has shown that Tong's theorem is best possible up to factors of $\log X$. Recently, Baluyot and Castillo \cite{baluyot_sign_2023} showed that, assuming the Riemann Hypothesis, Tong's theorem is also sharp for $k = 3$ (up to factors of $\log X$).

\section{Background and useful lemmas}
The primary tool leading to improvements in \cite{kolpakova_2011} over \cite{ivic_ouellet_1989} are lower bounds on the Carlson exponent $m(\sigma)$. This is defined as the supremum of all numbers $m \ge 4$ such that 
\begin{equation}\label{m_defn}
\int_1^T|\zeta(\sigma + it)|^{m}\text{d}t \ll_{\varepsilon} T^{1 + \varepsilon}
\end{equation}
for any $\varepsilon > 0$. We first recall the following classical result of Ivi\'{c} \cite[Thm.\ 8.4]{ivic_2003} 
\begin{equation}\label{ivic_m_estimate}
m(\sigma) \ge \frac{24\sigma - 9}{(4\sigma - 1)(1 - \sigma)}, \qquad 1/2 \le \sigma < 1.
\end{equation}
This is an estimate of order $(1 - \sigma)^{-1}$ for $\sigma$ close to 1. The main result of \cite{kolpakova_2011} depends on a sharper estimate, of order $(1 - \sigma)^{-3/2}$. In particular, it was shown\footnote{\label{fnote3}In \cite{kolpakova_2011} the function $m(\sigma)$ is defined as the supremum of all numbers $m$ for which $\int_1^T|\zeta(\sigma + it)|^{2m}\text{d}t \ll_{\varepsilon} T^{1 + \varepsilon}$, i.e. half of the $m(\sigma)$ used in our exposition.} that 
\[
m(\sigma) \ge \frac{2}{13.35(1 - \sigma)^{3/2}} + 159.9,\qquad 1 - (31.2)^{-1} < \sigma < 1.
\]
We improve this estimate in Lemma \ref{m_lower_bound}. The argument we use relies on estimates of the Carlson abscissa $\sigma_k$. For $k > 0$, define $\sigma_k$ as the infimum of numbers $\sigma$ for which
\begin{equation}
\int_1^T|\zeta(\sigma+it)|^{2k}\text{d}t\ll T.
\end{equation}
As mentioned in \cite[p.\ 153]{titchmarsh_theory_1986}, an equivalent definition for $\sigma_k$ is the infimum of numbers $\sigma$ for which 
\begin{equation}\label{sigmak_defn2}
\int_1^T|\zeta(\sigma + it)|^{2k}\text{d}t \ll_{\varepsilon} T^{1 + \varepsilon}
\end{equation}
for any fixed $\epsilon > 0$. Furthermore, throughout let 
\begin{equation}\label{defmu}
\mu_k(\eta) :=\limsup_{T\rightarrow \infty}\frac{\log\left(\frac{1}{T}\int_1^T|\zeta(\eta+it)|^{2k}\text{d}t\right)}{\log T},\qquad 0 \le \eta \le 1.
\end{equation}

We can obtain upper bounds on $\sigma_k$ using upper bounds on $\mu_k$, via the following result, originally due to Carlson \cite{carlson_contributions_1922, carlson_contributions_1926}.  
\begin{lemma}\label{titchmarsh_thm_79}
For any $0 < \eta < 1$, we have 
\[
\sigma_k \le \max\left\{1 - \frac{1 - \eta}{1 + \mu_k(\eta)}, \frac{1}{2}, \eta\right\}.
\]
\end{lemma}
\begin{proof}
See e.g.\ Titchmarsh \cite[Thm.\ 7.9]{titchmarsh_theory_1986}.
\end{proof}
We make use of this result in the proof of Theorem \ref{alpha_theorem}. Lastly, for the proof of Theorem \ref{alpha_theorem}, we also require the following upper bound on $\beta_k$, due to Titchmarsh.

\begin{lemma}\cite[Thm.\ 12.5]{titchmarsh_theory_1986}\label{lemma2} 
For any integer $k \ge 2$, $\beta_k$ is equal to the lower bound of positive numbers $\sigma$ for which
\[
\int_{-\infty}^{\infty}\frac{|\zeta(\sigma+it)|^{2k}}{|\sigma+it|^2}\text{d}t<\infty.
\]
\end{lemma}

\section{Proof of Theorem \ref{alpha_theorem}}

The proof of Theorem \ref{alpha_theorem} relies primarily on an inductive argument used in \cite{kolpakova_2011}, which produces an upper bound on $\sigma_k$, and hence by extension a lower bound on $m(\sigma)$. Such a bound is crucial to the improvement of the $\alpha_k$ estimate in \cite{kolpakova_2011} over that in \cite{ivic_ouellet_1989}.

\begin{lemma}\label{induction_lemma_1}
Suppose $\zeta(\sigma + it) \ll_{\varepsilon} t^{B(1 - \sigma)^{3/2} + \varepsilon}$ for some $B > 0$, $\sigma \in [1/2, 1]$ and any $\varepsilon > 0$. Let $\theta\in [1/2, 1)$ be fixed and $k_0(\theta)$ be as defined in \eqref{k0k1_defn}. Then, for any fixed $k \ge k_0(\theta)$ and fixed $\varepsilon_0 > 0$ we have
\begin{equation}\label{inductive_assumption}
\sigma_k \le 1 - \frac{1}{(3(B + \varepsilon_0)(k - k_2))^{2/3}},
\end{equation}
where
\begin{equation}\label{k2_defn}
k_2(\theta) := k_0(\theta) - \frac{1}{3(B + \varepsilon_0)(1 - \theta)^{3/2}}.
\end{equation}
\end{lemma}
\begin{proof}
For $r \ge k_0(\theta)$, let $P(r)$ denote the proposition
\[
P(r):\qquad \sigma_r \le 1 - \frac{1}{(3(B + \varepsilon_0)(s - k_2(\theta)))^{2/3}} \quad\text{for all}\quad k_0(\theta) \le s \le r.
\]
It thus suffices to show that $P(k)$ holds. We first show that $P(k_0)$ holds. By \eqref{ivic_m_estimate}, we have, for any $\theta \in [1/2, 1)$,
\begin{equation}
\int_1^T|\zeta(\theta + it)|^{\frac{24\theta - 9}{(4\theta - 1)(1 - \theta)}}\text{d}t = \int_1^T|\zeta(\theta + it)|^{2k_0(\theta)}\text{d}t \ll_{\varepsilon} T^{1 + \varepsilon},
\end{equation}
and hence $\sigma_{k_0(\theta)} \le \theta$ by definition \eqref{sigmak_defn2}. However, $k_0(\theta)$ and $k_2(\theta)$ are chosen so that
\[
\theta = 1 - \frac{1}{(3(B + \varepsilon_0)(k_0(\theta) - k_2(\theta)))^{2/3}},
\]
i.e. $P(k_0)$ holds by definition.

Assume now that $P(r)$ holds for some $k_0 \le r < k$. In particular, we have
\[
\sigma_{r} \le 1 - \frac{1}{(3(B + \varepsilon_0)(r - k_2))^{2/3}} = \eta_r,
\]
say. We will show that this implies $P(r + \Delta)$ where $\Delta$ is a fixed positive quantity. From the definition of $\sigma_r$, we have 
\[
\int_1^T|\zeta(\eta_r + it)|^{2r}\text{d}t \ll T. 
\]
This implies, for any fixed $\delta > 0$ and fixed $\varepsilon_1 > 0$ (in particular, independent of $r$)
\begin{align}
\int_1^T|\zeta(\eta_r + it)|^{2(r + \delta)}\text{d}t &\ll T^{2\delta B(1 - \eta_r)^{3/2} + \delta\varepsilon_1}\int_1^T|\zeta(\eta_r + it)|^{2r}\text{d}t \notag\\
&\ll T^{1 + 2\delta B(1 - \eta_r)^{3/2} + \delta\varepsilon_1}.  \label{induction_result_mu_k_delta}
\end{align}
Hence, using the definition of $\mu_{r+\delta}(\eta_r)$, we have
\[
\mu_{r + \delta}(\eta_r) \le 2\delta B(1 - \eta_r)^{3/2} + \delta\varepsilon_1 = \frac{2 \delta B}{3(B + \varepsilon_0)(r - k_2)} + \delta\varepsilon_1.
\]
However, from Lemma \ref{titchmarsh_thm_79},
\[
\sigma_{r + \delta} \le 1 - \frac{1 - \eta_r}{1 + \mu_{r + \delta}(\eta_r)} \le 1 - \frac{1}{(3(B + \varepsilon_0)(r - k_2))^{2/3}(1 + \frac{2\delta B}{3(B + \varepsilon_0)(r - k_2)} + \delta\varepsilon_1)}.
\]
The RHS is majorised by 
\[
1 - \frac{1}{(3(B + \varepsilon_0)(r - k_2 + \delta))^{2/3}}
\]
if and only if 
\begin{equation}\label{requirement_eqn}
\left(1 + x\right)^{2/3} \ge 1 + cx,
\end{equation}
where 
\[
c = \frac{2B}{3(B + \varepsilon_0)} + \varepsilon_1(r - k_2),\qquad x = \frac{\delta}{r - k_2}.
\]
We take 
\[
\varepsilon_1 = \frac{\varepsilon_0}{3(B + \varepsilon_0)(k - k_2)},
\]
which is a fixed positive quantity. This gives (for $r \le k$)
\[
c \le \frac{2}{3} - \frac{\varepsilon_0}{3(B + \varepsilon_0)}
\]
Since $c < 2/3$ for any $\varepsilon_0 > 0$, \eqref{requirement_eqn} is true for sufficiently small $x > 0$. In particular, we can choose a fixed $\Delta = \Delta(k_0, k_2)$ such that \eqref{requirement_eqn} is true for all $0 < \delta \le \Delta$ and $r \ge k_0 > k_2$. The result follows. 
\end{proof}

\begin{remark}
The requirement that $c < 2/3$ in the above lemma prevents us from replacing the constant of 3 in \eqref{inductive_assumption} with a smaller fixed constant. Lemma \ref{induction_lemma_1} thus represents a limiting case of the method. 
\end{remark}

\begin{lemma}\label{m_lower_bound}
Suppose $\zeta(\sigma + it) \ll_{\varepsilon} t^{B(1 - \sigma)^{3/2} + \varepsilon}$ uniformly for $1/2 \le \sigma \le 1$ and some $B > 0$. Then, for any $\varepsilon_0 > 0$ and $\theta \in [1/2, 1)$ we have 
\[
m(\sigma) \ge \frac{2}{3(B + \varepsilon_0)(1 - \sigma)^{3/2}} + 2k_2(\theta) 
\]
for all 
\[
1 - \frac{1}{(3(B + \varepsilon_0)(k_0(\theta) - k_2(\theta)))^{2/3}} \le \sigma < 1,
\]
where $k_0(\theta)$ is defined in \eqref{k0k1_defn} and $k_2(\theta)$ is defined in \eqref{k2_defn}.
\end{lemma}
\begin{proof}
From Lemma \ref{induction_lemma_1}, we have 
\[
\int_1^T|\zeta(g(m) + it)|^{2m}\text{d}t \ll T,\qquad g(m) := 1 - \frac{1}{(3(B + \varepsilon_0)(m - k_2))^{2/3}}
\]
Since $g$ is bijective, letting $\sigma = g(m/2)$ we have 
\[
\int_1^T|\zeta(\sigma + it)|^{2g^{-1}(\sigma)}\text{d}t \ll T,
\]
i.e. $m(\sigma) \ge 2g^{-1}(\sigma)$, as required. 
\end{proof}

\subsection{Bounds on $\alpha_k$} 
We now proceed to the proof of the first part of Theorem \ref{alpha_theorem}. Throughout, assume that $x$ is half an odd integer. Applying Perron's formula (see e.g. Titchmarsh \cite[Lem.\ 3.12]{titchmarsh_theory_1986}), we obtain, for any $c > 0$, $\sigma + c > 1$ and $x$ half an odd integer,
\[
\sum_{n \le x}d_k(n) = \frac{1}{2\pi i}\int_{c - iT}^{c + iT}\zeta^k(s)\frac{x^s}{s}\text{d}s + O_{\varepsilon}\left(\frac{x^c}{T(\sigma + c - 1)} + \frac{x^{1 + \varepsilon}}{T}\right),
\]
where we have used $d_k(n) \ll_{\varepsilon} n^{\varepsilon}$. We choose $c = 1 + 1/\log x$ and $\sigma > 0$, so that the error term is 
\[
\ll_{\varepsilon} \frac{x^{1 + \varepsilon}}{T},
\]
where $\varepsilon>0$ is arbitrarily small. Applying the residue theorem to the integral, we obtain
\begin{equation*}
    \begin{aligned}
     &\frac{1}{2\pi i}\int_{c - iT}^{c + iT} \zeta^k(s)\frac{x^s}{s}\text{d}s = xP_{k - 1}(\log x) \\
     &\qquad\qquad\qquad +\frac{1}{2\pi i}\left(\int_{c - iT}^{\beta - iT}+ \int_{\beta - iT}^{\beta - ih}+ \int_{E_3}+ \int_{\beta + ih}^{\beta + iT} + \int_{\beta + iT}^{c + iT}\right)\zeta^k(s)\frac{x^s}{s}\text{d}s\\
     &\qquad= J_1 + J_2 + J_3 + J_4 + J_5 + x P_{k - 1}(\log x),
    \end{aligned}
\end{equation*}
where $x P_{k - 1}(\log x)$ is the residue of the integrand $\zeta^k(s) x^s / s$ at the point $s = 1$ and $E_3$ is an arc of radius $1/2$ centered at the point $z_0=1$ from the point $\beta-ih$ to the point $\beta +ih$, with $0 < h < 1/2$, in the clockwise direction. Figure \ref{contour_pic} displays the contour used. It follows that 
\begin{equation}\label{delta_equation}
\Delta_k(x)= J_1 + J_2 + J_3 + J_4 + J_5 + O_{\varepsilon}\left(\frac{x^{1 + \varepsilon}}{T}\right).
\end{equation}

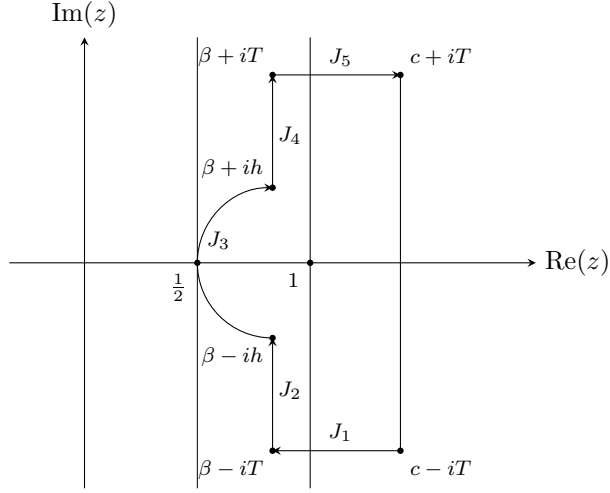
\begin{figure}
\centering 
\begin{tikzpicture}[>=stealth]
        \draw [->](-3,-3) -- (-3,3) node[above]{$\operatorname{Im}(z)$};
        \draw [->](-4,0) -- (3,0) node[right]{$\operatorname{Re}(z)$};
        \draw[black, -](1.2,-2.5) -- (1.2,2.5)node[below]{};
        \draw[black, ->](-0.5,-2.5) -- (-0.5,-1)node[above]{};
        \draw[black, ->](-0.5,1) -- (-0.5,2.5)node[above]{};
        \draw[black, ->](-0.5,2.5) -- (1.2,2.5)node[right]{};
        \draw[black, <-](-0.5,-2.5) -- (1.2,-2.5)node[left]{};
        \draw[black] (-1.5,0) arc (180:270:1);
        \draw[black,<-] (-0.5,1) arc (90:180:1);
        \draw[thin, -](0,-3) -- (0,3);
        \filldraw [black] (0,0) circle (1pt);
        \draw[ultra thin, -](-1.5,-3) -- (-1.5,3);
        \filldraw [black] (-1.5,0) circle (1pt);
        \node[below left =0.5pt of {(0,0)}, outer sep=0.5pt] {\fontsize{8pt}{10pt}\selectfont $1$};
        \node[below left =0.5pt of {(-1.5,0)}, outer sep=0.5pt] {\fontsize{8pt}{10pt}\selectfont $\frac{1}{2}$};
        \filldraw [black] (1.2,2.5) circle (1pt);
        \node[above right =0.2pt of {(1.2,2.5)}, outer sep=0.2pt] {\fontsize{8pt}{10pt}\selectfont $c+iT$};
        \filldraw [black] (1.2,-2.5) circle (1pt);
        \node[below right =0.2pt of {(1.2,-2.5)}, outer sep=0.2pt] {\fontsize{8pt}{10pt}\selectfont $c-iT$};
        \filldraw [black] (-0.5,-2.5) circle (1pt);
        \node[below left =0.1pt of {(-0.5,-2.5)},outer sep=0.1pt] {\fontsize{8pt}{10pt}\selectfont $\beta-iT$};
        \filldraw [black] (-0.5,-1) circle (1pt);
        \node[below left =0.1pt of {(-0.5,-1)},outer sep=0.1pt] {\fontsize{8pt}{10pt}\selectfont $\beta-ih$};
        \filldraw [black] (-0.5,1) circle (1pt);
        \node[above left =0.1pt of {(-0.5,1)},outer sep=0.1pt] {\fontsize{8pt}{10pt}\selectfont $\beta+ih$};
        \filldraw [black] (-0.5,2.5) circle (1pt);
        \node[above left =0.1pt of {(-0.5,2.5)},outer sep=0.1pt] {\fontsize{8pt}{10pt}\selectfont $\beta+iT$};
\node[above =0.1pt of {(0.4,-2.5)},outer sep=0.1pt] {\fontsize{8pt}{10pt}\selectfont $J_1$};
\node[right =0.1pt of {(-0.55,-1.7)},outer sep=0.1pt] {\fontsize{8pt}{10pt}\selectfont $J_2$};
\node[right =0.1pt of {(-1.5,0.3)},outer sep=0.1pt] {\fontsize{8pt}{10pt}\selectfont $J_3$};
\node[right =0.1pt of {(-0.55,1.7)},outer sep=0.1pt] {\fontsize{8pt}{10pt}\selectfont $J_4$};
\node[above =0.1pt of {(0.4,2.5)},outer sep=0.1pt] {\fontsize{8pt}{10pt}\selectfont $J_5$};
\end{tikzpicture}
\caption{Contour taken in Theorem \ref{alpha_theorem}.}\label{contour_pic}
\end{figure}

We estimate separately the moduli of the integrals $J_i$. Using the relation \eqref{heathbrown_assumption}\footnote{In the treatment of \cite{kolpakova_2011}, the estimate $|\zeta(\sigma + it)| \ll t^{B(1 - \sigma)^{3/2}}$ is used in this step. See also footnote \ref{fnote2}.} to estimate $|\zeta(\sigma +iT)|$ for $\sigma\in[\beta,1]$, and the estimate $|\zeta(\sigma +iT)|\ll \log^{2/3} T$ for $\sigma\in (1,c]$, we obtain
\begin{align*}
|J_1| = |J_5| &\ll \int_{\beta}^1\frac{|\zeta(\sigma + iT)|^kx^\sigma}{T}\text{d}\sigma + \int_1^{c}\frac{|\zeta(\sigma + iT)|^k x^{\sigma}}{T}\text{d}\sigma \\
&\ll_\varepsilon \int_\beta^1 x^{\sigma} T^{-1 + kB(1 - \sigma)^{3/2}+\varepsilon}\text{d}\sigma + \int_1^c x^c T^{-1}\log^{2k/3} T\text{d}\sigma\\
&= T^{-1+\varepsilon}\int_\beta^1 \exp(f(\sigma))\text{d}\sigma + xT^{-1}\frac{\log^{2k/3 }T}{\log x},
\end{align*}
where
\[
f(\sigma) = \sigma \log x + kB(1 - \sigma)^{3/2}\log T,
\]and, to estimate the second integral, we used the fact that $x^c=x^{1+1/\log x} = ex$. Since
\[
f''(\sigma) = \frac{3kB\log T}{4\sqrt{1 - \sigma}} > 0,
\]
 the function $f(\sigma)$ is convex, and hence $f(\sigma) \le \max\{f(\beta), f(1)\}$ for all $\beta \le \sigma \le 1$. It follows that
\begin{equation}\label{int15}
    \begin{aligned}
J_1 + J_5 &\ll_\varepsilon T^{-1+\varepsilon}\int_\beta^1 \max\left\{x, x^{\beta}T^{kB(1 - \beta)^{3/2}}\right\}\text{d}\sigma + \frac{x}{T}\frac{\log^{2k/3 }T}{\log x} \\&\ll_\varepsilon T^{\varepsilon}\left(x^{\beta }T^{kB(1 - \beta)^{3/2}-1}+\frac{x}{T}\right).
    \end{aligned}
\end{equation}
Now, we estimate the integral $J_3$. Using the analytic continuation of $\zeta(s)$ for $\sigma>-1$ (see \cite[(2.1.4)]{titchmarsh_theory_1986}) and the fact that $|s - 1| = 1/2$ and $|s| \ge 1/2$ on $E_3$, we have
\begin{equation}
J_3 \ll x^{\beta}\int_{E_3}\frac{\text{d}s}{|s-1|^k |s|} \ll x^{\beta},
\end{equation}
since the integral is convergent. 

Finally, we estimate the integrals $J_2$ and $J_4$ the same way. Let 
\begin{equation}\label{m0_defn}
m_0(\beta) := \frac{2}{3(B + \varepsilon_0)(1 - \beta)^{3/2}} + 2k_2(\theta),
\end{equation}
so that, by Lemma \ref{m_lower_bound},
\[
\int_{Z}^{2Z}|\zeta(\beta + it)|^{m_0(\beta)}\text{d}t \ll_{\varepsilon} Z^{1 + \varepsilon}.
\]
Then, for all $\beta$ such that $m_0(\beta) \le k$, we have, for $1 \ll Z \ll T$, \footnote{\label{fnote5}
In \cite{kolpakova_2011} the estimate $\int_{T_n}^{T_{n + 1}}|\zeta(\beta + it)|^{k - 2m(\beta)}|\zeta(\beta + it)|^{2m(\beta)}\text{d}t$ is used, instead of the exponents $k - m(\beta)$ and $m(\sigma)$ respectively, since the definition of $m(\sigma)$ in \cite{kolpakova_2011} differs from ours. See also footnote \ref{fnote3}.}
\begin{align*}
\int_{\beta + iZ}^{\beta + 2iZ}\zeta^k(s)\frac{x^s}{s}\text{d}s &\ll \frac{x^{\beta}}{Z}\int_{Z}^{2Z}|\zeta(\beta + it)|^k\text{d}t \\
&\ll_\varepsilon x^{\beta}Z^{B(k - m_0(\beta))(1 - \beta)^{3/2} -1+\varepsilon}\int_{Z}^{2Z}|\zeta(\beta + it)|^{m_0(\beta)}\text{d}t\\
&\ll_{\varepsilon} x^{\beta} Z^{B(k - m_0(\beta))(1 - \beta)^{3/2}+\varepsilon} .
\end{align*}
Using a dyadic division and summing over $O(\log T)$ intervals, we get 
\begin{equation}\label{J4_estimate}
J_4 = \int_{\beta + iT/2}^{\beta + iT} + \int_{\beta + iT/4}^{\beta + iT/2} + \cdots + \int_{\beta + ih}^{\beta + 2ih} \ll_\varepsilon x^{\beta}T^{B(k - m_0(\beta))(1 - \beta)^{3/2}+\varepsilon},
\end{equation}
for any $\varepsilon>0$, and similarly for $J_2$. At this point, combining all the estimates, we have
\begin{equation}\label{divisor_bound}
\Delta_k(x) \ll_{\varepsilon} T^{\varepsilon}\left(x^{\beta}T^{B(k - m_0(\beta))(1 - \beta)^{3/2}}+\frac{x}{T}\right).
\end{equation}
In order to optimise the above estimate, we choose $\beta$ so as to balance the main terms
\[
x^{\beta}T^{B(k - m_0(\beta))(1 - \beta)^{3/2}},\qquad \frac{x}{T}.
\]
The two terms are equal if $x^{f(\beta)} = T$, where 
\begin{equation}\label{f_beta_bound}
f(\beta) = \frac{1 - \beta}{1 + B(k - m_0(\beta))(1 - \beta)^{3/2}} = \frac{1 - \beta}{1 + B(k - 2k_2)(1 - \beta)^{3/2} - \frac{2B}{3(B + \varepsilon_0)}}.
\end{equation}
We choose
\[
\beta = 1 - \left(\frac{2 - \frac{4B}{3(B + \varepsilon_0)}}{B(k - 2k_2(\theta))}\right)^{2/3}
\]
so as to maximise the RHS of \eqref{f_beta_bound}. Verifying that $x^{\beta} \ll x^{1 - f(\beta)}$, this gives 
\[
\Delta_k(x) \ll_{\varepsilon} x^{1 - f(\beta) + \varepsilon} \ll_{\varepsilon} x^{1 - c_0'(kB)^{-2/3} + \varepsilon}
\] 
where 
\[
c_0' = \left(\frac{2 - \frac{4B}{3(B + \varepsilon_0)}}{1 - 2k_2/k}\right)^{2/3},
\]
where $\varepsilon_0>0$ is an arbitrarily small constant. However, since $k$ is fixed, $c_0' = c_0 + \varepsilon$ for some $\varepsilon > 0$. The first part of Theorem \ref{alpha_theorem} follows.

To complete the proof of the first part of Theorem \ref{th1}, we choose $\theta = 0.839427\ldots$ so as to maximise the value of $k_1(\theta)$ when $B = 8\sqrt{15}/63$, which gives $k_1(\theta) = 4.187\ldots$ and $k_0(\theta) = 14.72\ldots$.

\subsection{Bounds on $\beta_k$}
We now turn our attention to the constant $c_1$. By Lemma \ref{lemma2}, it suffices to prove that the integral 
\[
\int_{-\infty}^{\infty}\frac{|\zeta(\sigma+it)|^{2k}}{|\sigma+it|^2}\text{d}t
\]
converges for all
\begin{equation}\label{sigma_interval}
\left(\frac{5}{6B(k - k_1(\theta))}\right)^{2/3}<\sigma<1.
\end{equation}
By symmetry, it suffices to show that the integral
\begin{equation}\label{intconv}
\int_{1}^{\infty}\frac{|\zeta(\sigma+it)|^{2k}}{|\sigma+it|^2}\text{d}t
\end{equation}
converges, since the integral converges on $[-1, 1]$ and $\sigma < 1$. 

For any $1 \ll Z \ll T$, we have 
\begin{equation}\label{estc1}
\begin{aligned}
\int_{Z}^{2Z}\frac{|\zeta(\sigma+it)|^{2k}}{|\sigma+it|^2}\text{d}t &\leq \frac{1}{Z^2}\int_{Z}^{2Z}|\zeta(\sigma + it)|^{2k - m_0(\sigma)}|\zeta(\sigma + it)|^{m_0(\sigma)}\text{d}t\\
&\ll_{\varepsilon} Z^{-2 + B(2k - m_0(\sigma))(1 - \sigma)^{3/2} + \varepsilon}\int_{Z}^{2Z}|\zeta(\sigma+it)|^{m_0(\sigma)}\text{d}t\\
&\ll_\varepsilon Z^{-1+ B(2k - m_0(\sigma))(1 - \sigma)^{3/2} + \varepsilon}
\end{aligned}
\end{equation}
where $m_0(\sigma)$ is defined in \eqref{m0_defn}\footnote{This expression differs from that in Kolpakova's treatment, due to the difference in definition of $m(\sigma)$. See also footnote \ref{fnote5}.}, and since $m_0(\sigma)\le m(\sigma)$ by Lemma \ref{m_lower_bound}, 
\[
\int_{Z}^{2Z}|\zeta(\sigma+it)|^{m_0(\sigma)}\text{d}t\ll_{\varepsilon} Z^{1 + \varepsilon},
\]
for all $\varepsilon > 0$ arbitrarily small. Using a dyadic division, we thus have 
\begin{equation}\label{est2}
\int_{1}^{T}\frac{|\zeta(\sigma+it)|^{2k}}{|\sigma+it|^2}\text{d}t = \int_{T/2}^{T} + \int_{T/4}^{T/2} + \cdots \ll_\varepsilon T^{-1+B(2k-m_0(\sigma))(1-\sigma)^{3/2}+\varepsilon}.   
\end{equation}
The exponent of the parameter $T$ of the RHS is
\begin{equation*}
\begin{aligned}
&-1+B(1-\sigma)^{3/2}(2k-m_0(\sigma))\\
&\qquad = -1 + 2B(k - k_1(\theta))(1 - \sigma)^{3/2} - \frac{2B}{3(B + \varepsilon_0)}.
\end{aligned}
\end{equation*}
This is negative if 
\[
1 - \left(\frac{5B + 3\varepsilon_0}{6B(B + \varepsilon_0)(k - k_1(\theta))}\right)^{2/3} < \sigma < 1
\]
which is the same as \eqref{sigma_interval} since $\varepsilon_0$ is arbitrarily small.  Therefore, the exponent on the RHS of \eqref{est2} is non-positive for $\varepsilon$ sufficiently small. Taking $T \to \infty$, the second part of Theorem \ref{alpha_theorem} then follows from Lemma \ref{lemma2}.

To conclude the proof of Theorem \ref{th1}, we choose $\theta = 0.839427\ldots$ and $B = 8\sqrt{15}/63$ as before. 

\section{Proof of Theorem \ref{ctheorem}}
We begin by refining an exponential sum estimate due to Heath-Brown \cite{heathbrown_new_2017}, which are in turn based on the bounds on the Vinogradov mean value integral proved in \cite{wooley_cubic_2016, bourgain_proof_2016}.
\begin{lemma}\label{imp_exponential_sum_est}
Let $\rho = \log N / \log t$ and $N \le N' \le 2N$. Then for $\rho \ge 3$,
\[
\sum_{N < n \le N'}n^{-it} \ll_{\varepsilon} N^{1 - (1 - 3\rho^{-1})\rho^{-2} + \varepsilon},
\]
for any $\varepsilon > 0$. 
\end{lemma}
Note that \cite[Thm.\ 4]{heathbrown_new_2017} has $49/80$ in place of $1 - 3/\rho$. Lemma \ref{imp_exponential_sum_est} is sharper for $\rho \ge 8$. The result in fact holds for all $\rho > 0$, since for $\rho < 3$ it is implied by the trivial bound. 
\begin{proof}
For $k \ge 2$, let $\rho_k := (k^2 + 1) / (k + 1)$. Applying \cite[Thm.\ 1]{heathbrown_new_2017}, combined with \cite[(17)]{heathbrown_new_2017}, we have, for $\rho_{k - 1} \le \rho \le \rho_k$,
\[
\sum_{N < n \le N'} n^{-it} \ll_{\varepsilon} N^{1 + \phi(\rho) + \varepsilon},\qquad \phi(\rho) = A_k \rho + B_k
\]
where
\[
A_k := \frac{2}{(k - 1)^2(k + 2)},\qquad B_k := -\frac{3k^2 - 3k + 2}{k(k - 1)^2(k + 2)}.
\]
Observe that for $k \ge 2$, 
\[
\phi(\rho_k) = -c_k\rho_k^{-2},\qquad c_k := \frac{(k^2 + 1)^2}{k(k + 1)^3}.
\]
Since $c_k$ is decreasing in $k$ for $k \ge 2$, we also have $\phi(\rho_{k + 1}) = -c_{k + 1}\rho_{k + 1}^{-2} \le -c_k\rho_{k + 1}^{-2}$. For each fixed $k$, $-c_k\rho^{-2}$ is concave while $\phi(\rho)$ is affine, so it follows that for $k \ge 2$ we have
\[
\phi(\rho) \le -c_k \rho^{-2},\qquad \rho_k \le \rho \le \rho_{k + 1}.
\]
However, for $\rho_k \le \rho \le \rho_{k + 1}$ and $k \ge 2$,
\[
c_k > 1 - 3\frac{k + 2}{(k + 1)^2 + 1} = 1 - \frac{3}{\rho_{k + 1}} \ge 1 - \frac{3}{\rho},
\]
where the first inequality is verified via a routine calculation. Therefore
\[
\sum_{N < n \le N'} n^{-it} \ll_{\varepsilon} N^{1 + \phi(\rho) + \varepsilon} \ll N^{1 - (1 - 3\rho^{-1})\rho^{-2} + \varepsilon},\qquad \rho \ge \rho_2 = \frac{5}{3},
\]
as required.
\end{proof}
Lemma \ref{imp_exponential_sum_est} leads to the following improved estimate of Carlson's abscissa, $\sigma_k$. 
\begin{lemma}\label{imp_carleson_abscissa}
Let $\delta > 0$ be fixed and sufficiently small. Then, there exists an absolute constant $A$ such that for all integers $k \ge A\delta^{-3}$, we have 
\[
\sigma_k \le 1 - \left(\frac{3}{2^{2/3}} - \delta\right) k^{-2/3}.
\]
\end{lemma}
\begin{proof}
It suffices to show that for any $\delta > 0$, if $k \ge A\delta^{-3}$ then
\[
\int_T^{2T}|\zeta(\sigma + it)|^{2k}\text{d}t \ll_{\varepsilon} T^{1 + \varepsilon}\quad\text{for}\quad \sigma \ge 1 - \left(\frac{3}{2^{2/3}} - \delta\right) k^{-2/3},
\]
since the desired result then follows by summing the integrals taken over $[T/2, T]$, $[T/4, T/2]$ and so on. Note that the parameter $\delta$ used here is different to that in the proof of Theorem \ref{alpha_theorem}. 
Using Minkowski's inequality, for $T \le t \le 2T$ we have 
\[
\int_T^{2T}\left|\sum_{n \le T^{1/2}} n^{-\sigma- it}\right|^{2k}\text{d}t \le \int_T^{2T}\left|\sum_{n \le T^{1/k}}n^{-\sigma-it}\right|^{2k}\text{d}t + \int_T^{2T}\left|\sum_{T^{1/k} < n \le T^{1/2}}n^{-\sigma- it}\right|^{2k}\text{d}t.
\]
For $\sigma > 0$,
\begin{equation}\label{sum_part_1}
\int_T^{2T}\left|\sum_{n \le T^{1/k}}n^{-\sigma-it}\right|^{2k}\text{d}t \ll \int_T^{2T}\left|\sum_{n \le T^{1/k}}n^{-it}\right|^{2k}\text{d}t \le \int_T^{2T}\left|\sum_{n \le T}d_k(n)n^{-it}\right|^{2}\text{d}t \ll_{\varepsilon} T^{1 + \varepsilon},
\end{equation}
where in the last inequality we have used $d_k(n) \ll_{\varepsilon} n^{\varepsilon}$ and the mean value theorem for Dirichlet polynomials.\footnote{Since $d_{k}(n)=\boldsymbol{1}^{*k}$, $d(n) \ll_{\varepsilon} n^{\varepsilon}$ by a classical result, and the convolution of two arithmetic functions that are of at most subpolynomial growth is still of at most subpolynomial growth}
To bound the second integral, consider 
\[
I(T, N) := \int_T^{2T}\left|\sum_{N \le n \le 2N}n^{- \sigma - it}\right|^{2k}\text{d}t,\qquad T^{1/k} < N \le T^{1/2}. 
\]
Let 
\[
\rho := \frac{\log T}{\log N},\qquad \ell := \lfloor \rho \rfloor
\]
so that $\rho \in [2, k]$ and
\begin{equation}\label{ell_lower_bound}
\rho - 1 < \ell \le \rho \le k.
\end{equation}
We have 
\begin{align*}
\int_T^{2T}\left|\sum_{N < n \le 2N}n^{-it}\right|^{2\ell}\text{d}t = \int_T^{2T}\left|\sum_{N^{\ell} < n \le (2N)^{\ell}}a_n n^{-it}\right|^{2}\text{d}t
\end{align*}
where 
\begin{equation}\label{a_n_bound}
a_n := \sum_{\substack{N < n_1, \ldots n_\ell \le 2N\\ n_1\cdots n_\ell = n}}1 \le d_k(n) \ll_{\varepsilon} n^{\varepsilon}
\end{equation}
for any $\varepsilon > 0$. Furthermore,
\begin{equation}\label{a_n_count_bound}
\sum_{N^{\ell} < n \le (2N)^{\ell}} \textbf{1}\{a_n > 0\} \le N^\ell
\end{equation}
by counting the maximum number of terms in the product expansion. Combining \eqref{a_n_bound} and \eqref{a_n_count_bound} and applying the mean-value theorem for Dirichlet polynomials,  
\begin{equation}\label{segment_1_bound}
\int_T^{2T}\left|\sum_{N^\ell < n \le (2N)^\ell}a_n n^{-it}\right|^2\text{d}t \ll T\sum_{N^\ell < n \le (2N)^{\ell}}|a_n|^2 \ll_{\varepsilon} TN^{\ell + \varepsilon}
\end{equation}
for any $\varepsilon > 0$. Additionally, by Lemma \ref{imp_exponential_sum_est} if $\rho \ge 3$ and the trivial bound if $2 \le \rho < 3$,
\begin{equation}\label{segment_2_bound}
\left|\sum_{N < n \le 2N}n^{-it}\right|^{2(k - \ell)} \ll_{\varepsilon} N^{2(k - \ell)(1 - (1 - 3\rho^{-1})/\rho^2) + \varepsilon}
\end{equation}
uniformly for $T \le t \le 2T$. Note that in this step we have used $\ell \le k$, a bound that originates from using of the mean value theorem for Dirichlet polynomials to cover the range $\rho \in [k, \infty)$. Hence, combining \eqref{segment_1_bound} and \eqref{segment_2_bound}, and using $N = T^{1/\rho}$ and partial summation,
\begin{align*}
I(T, N) &\le N^{-2k\sigma}\max_{T \le t \le 2T}\left|\sum_{N < n \le 2N}n^{-it}\right|^{2(k - \ell)}\int_T^{2T}\left|\sum_{N < n \le 2N}n^{-it}\right|^{2\ell}\text{d}t\\
&\ll_{\varepsilon} T^{1 + \varepsilon}N^{2(k - \ell)(1 - (1 - 3\rho^{-1})\rho^{-2}) + \ell - 2k\sigma + \varepsilon}\\
&= T^{1 + 2(k - \ell)(1 - (1 - 3\rho^{-1})\rho^{-2})/\rho + \ell/\rho - 2k\sigma /\rho + \varepsilon}.
\end{align*}
However, since $\rho - 1 < \ell \le \rho$ by \eqref{ell_lower_bound}, the exponent of $T$ is majorised by 
\[
h(\rho) = \frac{2(k - \rho + 1)(1 - \rho^{-2} + 3\rho^{-3}) - 2k\sigma}{\rho} + 2.
\]
Choosing
\[
\sigma = 1 - \alpha k^{-2/3}
\]
for some constant $\alpha > 0$ and substituting, we find, via a direct evaluation
\[
h(\rho) = 2\alpha \frac{k^{1/3}}{\rho} - 2\frac{k + 4}{\rho^3} + 6\frac{k + 1}{\rho^4} + \frac{2}{\rho^2} + \frac{2}{\rho}.
\]
The function $h(\rho)$ is maximised on $[2, k]$ by the choice $\rho$ satisfying
\[
(\alpha k^{1/3} + 1)\rho^3 + 2\rho^2 - 3(k + 4)\rho + 12(k + 1) = 0.
\]
The unique positive solution is $\rho^*$ satisfying
\[
\rho^* = \sqrt{\frac{3}{\alpha}}k^{1/3} + O(1),\qquad k\to\infty.
\]
This gives, after some simplification,
\[
h(\rho) \le h(\rho^*) = \frac{4}{\sqrt{27}}\alpha^{3/2} + O(k^{-1/3}),
\]
where the implied constant depends only on $\alpha$. Therefore, for any sufficiently small fixed $\delta > 0$, by choosing 
\[
\alpha = \frac{3}{2^{4/3}} - \delta
\]
we have (since $(x - \delta)^{3/2} < x^{3/2} - \delta$ for sufficiently small $\delta > 0$)
\[
h(\rho) \le \frac{4}{\sqrt{27}}\left(\frac{3}{2^{4/3}} - \delta\right)^{3/2} + O(k^{-1/3}) < 1 - \frac{4}{\sqrt{27}}\delta + O(k^{-1/3}) 
\]
which is no greater than 1 if $k \ge A\delta^{-3}$ for sufficiently large constant $A$. Therefore, for $\sigma \ge 1 - \alpha k^{-2/3}$ and $k \ge A\delta^{-3}$, we have 
\[
I(T, N) \ll_{\varepsilon} T^{1 + \varepsilon}.
\]
Choosing $N = T^{1/2} / 2$, $T^{1/2} / 4$, $T^{1/2}/8$, $\ldots$ and via repeated application of Minkowski's inequality,
\begin{align}
\int_T^{2T}\left|\sum_{T^{1/k} \le n \le T^{1/2}}n^{-\sigma - it}\right|^{2k}\text{d}t &\ll I\left(T, \frac{T^{1/2}}{2}\right) + I\left(T, \frac{T^{1/2}}{4}\right) + I\left(T, \frac{T^{1/2}}{8}\right) + \cdots \notag\\
&\ll_{\varepsilon} T^{1 + \varepsilon}. \label{sum_part_2}
\end{align}
Therefore, combining \eqref{sum_part_1} and \eqref{sum_part_2},
\begin{equation}\label{approx_sum_1}
\int_T^{2T}\left|\sum_{1 \le n \le T^{1/2}}n^{-\sigma - it}\right|^{2k}\text{d}t \ll_{\varepsilon} T^{1 + \varepsilon},\qquad \sigma \ge 1 - \left(\frac{3}{2^{4/3}}  - \delta\right) k^{-2/3}.
\end{equation}
Via partial summation, we thus also have (for the same range of $\sigma$) 
\begin{equation}\label{approx_sum_2}
\int_T^{2T}\left|\sum_{1 \le n \le T^{1/2}}n^{\sigma - 1 + it}\right|^{2k}\text{d}t \ll_{\varepsilon} T^{1/2 - \sigma}T^{1 + \varepsilon}.
\end{equation}
Using the approximate functional equation for $\zeta(s)$, for $s = \sigma + it$ and $t \in [T, 2T]$ we have 
\begin{equation}\label{approx_func_equation}
\zeta(s) = \sum_{1 \le n \le T^{1/2}}n^{-s} + \chi(1 - s)\sum_{1 \le n \le T^{1/2}}n^{1-s} + o(1).
\end{equation}
as $t \to \infty$, where 
\[
\chi(s) = \pi^{1/2 - s}\frac{\Gamma(s / 2)}{\Gamma((1 - s) /2)}
\]
so that $\chi(1 - s) \ll T^{\sigma - 1/2}$. Therefore, combining \eqref{approx_sum_1}, \eqref{approx_sum_2} and \eqref{approx_func_equation}, and using Minkowski's inequality, we finally have 
\[
\int_T^{2T}|\zeta(\sigma + it)|^{2k}\text{d}t \ll_{\varepsilon} T^{1 + \varepsilon},\qquad \sigma \ge 1 - \left(\frac{3}{2^{4/3}} - \delta\right) k^{-2/3},
\]
from which the desired result follows. 
\end{proof}

Next, we require a more precise estimate of $\zeta(s)$ close to $\sigma = 1$. Since the proof shares some similarities with Lemma \ref{imp_carleson_abscissa}, we shall keep the exposition terse where possible. 

\begin{lemma}\label{zeta_bound_lem}
If $\delta > 0$ is fixed, then for all $k \ge A\delta^{-3}$, where $A$ is an absolute constant, we have 
\[
\zeta(1 - \alpha k^{-2/3} + it) \ll_{\varepsilon} t^{(2\alpha^{3/2}/3^{3/2} + \delta)/k + \varepsilon}
\]
for any $\varepsilon > 0$. 
\end{lemma}
\begin{proof}
Throughout let $\sigma = 1 - \alpha k^{-2/3}$, $\rho = \log t / \log N$ and $N < N' \le 2N$. We use Lemma \ref{imp_exponential_sum_est} for $\rho \ge 3$ and the trivial bound for $2 \le \rho < 3$ to obtain 
\[
\sum_{N < n \le N'}n^{-\sigma - it} \ll_{\varepsilon} N^{1 - \sigma - (1 - 3\rho^{-1})\rho^{-2} + \varepsilon} = t^{h(\rho) + \varepsilon}
\]
where 
\[
h(\rho) := \frac{\alpha k^{-2/3}}{\rho} - \frac{(1 - 3\rho^{-1})}{\rho^{3}}.
\]
The function $h(\rho)$ is maximised by $\rho$ satisfying
\[
\alpha k^{-2/3} \rho^3 - 3\rho + 12 = 0,
\]
the unique positive solution of which satisfies 
\[
\rho^* = \sqrt{\frac{3}{\alpha}}k^{1/3} + O(1),\qquad k\to\infty,
\]
where the implied constant is absolute. Substituting this into $h(\rho)$, we have 
\[
h(\rho) \le h(\rho^*) \le \frac{2\alpha^{3/2}}{3^{3/2}}k^{-1} + A'k^{-4/3}
\]
for some absolute constant $A' > 0$. Thus, for $k \ge A'^3\delta^{-3}$ we have 
\[
\sum_{N < n \le N'}n^{-\sigma - it} \ll_{\varepsilon} t^{(2\alpha^{3/2}/3^{3/2} + \delta)/k + \varepsilon}.
\]
Via a dyadic division, 
\[
\sum_{1 \le n \le t^{1/2}}n^{-\sigma -it} = \sum_{\frac{t^{1/2}}{2} < n \le t^{1/2}} + \sum_{\frac{t^{1/2}}{4} < n \le \frac{t^{1/2}}{2}} + \cdots \ll_{\varepsilon} t^{(2\alpha^{3/2}/3^{3/2} + \delta)/k + \varepsilon}.
\]
The desired result follows from partial summation and the approximate functional equation for $\zeta(s)$. 
\end{proof}
\subsection{Bounds on $\alpha_k$}
The proof of Theorem \ref{ctheorem} is similar to that of the first part of Theorem \ref{alpha_theorem}, except we use Lemma \ref{imp_carleson_abscissa} in place of Lemma \ref{induction_lemma_1}, and Lemma \ref{zeta_bound_lem} instead of \eqref{heathbrown_assumption}. Throughout, let $\delta > 0$ be fixed, $A$ be an absolute constant (not necessarily the same at each occurrence) and suppose $k \ge A\delta^{-3}$. Using Lemma \ref{imp_carleson_abscissa}, we have 
\begin{equation}\label{m1_defn}
m(\sigma) \ge m_1(\sigma) := 2\left(\frac{3}{2^{4/3}} - \delta\right)^{3/2}(1 - \sigma)^{-3/2}
\end{equation}
for all $\sigma$ satisfying
\begin{equation}\label{m1_condition}
m_1(\sigma) \ge A\delta^{-3},\qquad\text{i.e.}\qquad \sigma \ge 1 - A\delta^2.
\end{equation}
We largely follow the same argument as in the proof of Theorem \ref{ctheorem}. Throughout we take
\[
\beta = 1 - \alpha k^{-2/3} 
\]
for some constant $\alpha > 0$ to be chosen later, so that since $k \ge A\delta^{-3}$, condition \eqref{m1_condition} is satisfied. 

Following the argument leading up to \eqref{J4_estimate}, and making the necessary changes to make use of Lemma \ref{zeta_bound_lem} and replacing $m_0(\beta)$ with $m_1(\beta)$, we have
\[
|J_2| = |J_4| \ll_{\varepsilon} x^{\beta}T^{2/3^{3/2}(k - m_1(\beta))(1 - \beta)^{3/2} + \delta + \varepsilon},
\]
provided that $\beta$ is chosen so that $k - m_1(\beta) \ge 0$, i.e. that
\begin{equation}\label{beta_constraint}
\beta \le 1 - \left(\frac{3}{2^{2/3}} - 2^{2/3}\delta\right)k^{-2/3}.
\end{equation}
In addition, we also have  
\begin{align*}
|J_1| = |J_5| &\ll_{\varepsilon} T^{\varepsilon}\left(x^{\beta }T^{2k/3^{3/2}(1 - \beta)^{3/2} + \delta -1}+\frac{x}{T}\right),
\end{align*}
and $J_3 \ll x^{\beta}$ as before. This gives 
\begin{equation}\label{divisor_estimate_2}
\Delta_k(x) \ll_{\varepsilon} x^{\beta}T^{2/3^{3/2}(k - m_1(\beta))(1 - \beta)^{3/2} + \delta + \varepsilon} + x^{1 + \varepsilon}T^{-1 + \varepsilon}.
\end{equation}
 We once again choose $T$ to balance the terms in \eqref{divisor_estimate_2} (ignoring $\varepsilon$). This is achieved by letting $x^{f(\beta)} = T$, where 
\begin{equation}\label{f_beta_boundc}
\begin{split}
f(\beta) &= \frac{1 - \beta}{1 + \frac{2}{3^{3/2}}(k - m_1(\beta))(1 - \beta)^{3/2} + \delta}
\end{split}
\end{equation}
Choosing 
\[
\beta = 1 - \left(\frac{3}{2^{2/3}} - 2^{2/3}\delta\right)k^{-2/3},
\]
so that $m_1(\beta) = k$, we have 
\[
f(\beta) = \left(\frac{3/2^{2/3} - 2^{2/3}\delta}{1 + \delta}\right)k^{-2/3} >\left(\frac{3}{2^{2/3}} - 3\delta\right)k^{-2/3}.
\]
Therefore, for any $\varepsilon, \delta > 0$ and $k \ge A\delta^{-3}$, 
\[
\Delta_k(x) \ll_{\varepsilon} x^{1 - f(\beta) + \varepsilon} = x^{1 - (3/2^{2/3} - 3\delta)k^{-2/3} + \varepsilon}.
\]
The result follows from replacing $\delta$ with $\delta/3$.

\section*{Acknowledgements}
We would like to thank K. Ford, D. R. Heath-Brown and B. Kerr for their insightful comments upon the first version of this article. In particular, we are indebted to B. Kerr for the method behind the proof of Theorem \ref{ctheorem}. Thanks also to O.\ Bordelles for spotting some errors and to T. S. Trudgian for his support throughout the writing of this article. 

\printbibliography
\end{document}